\documentclass[12pt]{amsart}

\allowdisplaybreaks[1]
\usepackage{enumerate}
\usepackage{graphicx}
\usepackage{amsmath}
\usepackage{amssymb}
\usepackage{amsfonts}
\usepackage[dvipsnames]{xcolor}
\usepackage{hyperref}
\allowdisplaybreaks
 \newtheorem{theorem}{Theorem}[section]
 \newtheorem{corollary}[theorem]{Corollary}
 
 \newtheorem{proposition}[theorem]{Proposition}

\newtheorem{observation}[theorem]{Observation}
\theoremstyle{definition}

\theoremstyle{remark}

\newtheorem{fact*}{Fact}
\newtheorem{note}[theorem]{Note}
\DeclareMathOperator{\vspan}{span}

\DeclareMathOperator{\supp}{supp}

\newcommand{\hilbert}{\mathcal{H}}

\newcommand{\C}{\mathbb{C}}

\newcommand{\BH}{\mathcal{B}(\mathcal{H})}

\newcommand{\cc}[1]{\overline{#1}}

\newcommand{\norm}[1]{\left\Vert#1\right\Vert}

\newcommand{\ran}[1]{\operatorname{ran}#1}

\newcommand{\ad}{^\ast}
\newcommand{\inv}{^{-1}}

\newcommand{\til}{\raise.17ex\hbox{$\scriptstyle\mathtt{\sim}$}}

\newcommand{\ph}{\varphi}

\newcommand\la{\lambda}

\newcommand\beq{\begin{equation}}

\newcommand\eeq{\end{equation}}

\newcommand\bbm{\begin{bmatrix}}
\newcommand\ebm{\end{bmatrix}}
\newcommand{\bpm}{\left( \begin{smallmatrix}}
\newcommand{\epm}{\end{smallmatrix} \right)}
\numberwithin{equation}{section}

\newlength{\Mheight}
\newlength{\cwidth}

\newcommand{\dfn}[1]{{\bf #1}\index{#1}}

\newcommand{\Mn}{M_n(\C)}

\title[Nonlinear Stinespring]{Induced Stinespring factorization and the Wittstock support theorem}
\author[J. E. Pascoe]{
J. E. Pascoe$^\dagger$
}
\address{Department of Mathematics\\
1400 Stadium Rd\\
  University of Florida\\
 Gainesville, FL 32611}
\email[J. E. Pascoe]{pascoej@ufl.edu}
\thanks{$\dagger$ Partially supported by National Science Foundation DMS Analysis Grant 1953963}

\author[R. Tully-Doyle]{
Ryan Tully-Doyle$^\ddagger$
}
\address{Mathematics Department \\
1 Grand Ave \\
Cal Poly, SLO\\
San Luis Obispo, CA 93407}
\email[R. Tully-Doyle]{rtullydo@calpoly.edu}
\thanks{$\ddagger$ Partially supported by National Science Foundation DMS Analysis Grant 2055098}

\date{\today}


\begin{document}

\maketitle

\begin{abstract}
Given a pair of self-adjoint-preserving completely bounded maps on the same $C^*$-algebra, say that $\varphi \leq \psi$ if the kernel of $\varphi$ is a subset of the kernel of $\psi$ and $\psi \circ \varphi^{-1}$ is completely positive.
The \emph{Agler class} of a map $\varphi$ is the class of $\psi \geq \varphi.$
Such maps admit colligation formulae, and, in Lyapunov type situations, transfer function type realizations on the Stinespring coefficients of their Wittstock decompositions.
%Special cases of our general approach give various transfer function realizations, which in turn can be used to prove Herglotz and Nevanlinna type representations, and Nevanlinna-Pick interpolation theorems, including those on the disk, bidisk, free noncommutative semi-algebraic sets and operatorial semi-algebraic sets.
As an application, we prove that the support of an extremal Wittstock decomposition is unique.
\end{abstract}

\section{Introduction}

%\begin{figure}[h!]
%\includegraphics[scale=.1]{Locutus.jpg}
%\caption{Your interpolation theory will be assimilated}
%\end{figure}

First, we recall the following theorem.
\begin{theorem}[Nevanlinna-Pick interpolation theorem] \label{oldnevpick}
	Let $z_1,\ldots, z_n \in \mathbb{D}$ and $\la_1,\ldots, \la_n \in \mathbb{C}.$
	There is an analytic function $f: \mathbb{D} \rightarrow \cc{\mathbb{D}}$
	such that $f(z_i)=\la_i$ if and only if the matrix $\left[\frac{1-\cc{\la_i}\la_j}{1-\cc{z_i}z_j}\right]_{i,j}$
	is positive semidefinite.
\end{theorem}
See \cite{agmcbook} for a comprehensive reference on Pick interpolation. Importantly, the elementary so-called ``lurking isometry argument" gives that such an $f$ is of the form
	$$f(z) = a +b^*z (1-Dz)^{-1}c$$
where 
	$$\bbm a & b^* \\ c & D\ebm$$
is a unitary operator. Such a formula is often called a transfer function realization.

Let $\mathcal{M}, \mathcal{N}$ be $C^*$-algebras. Given map $\varphi:\mathcal{M}\rightarrow \mathcal{N}$
we define the $n$-th induced map $\varphi^{(n)}: M_n(\mathcal{M}) \rightarrow M_n(\mathcal{M})$
by entrywise evaluation as $$\varphi^{(n)}([m_{ij}]_{ij})=[\varphi(m_ij)]_{ij}.$$
We say that $\varphi$ is \dfn{completely bounded} if each $\varphi^{(n)}$ is a bounded linear map. We call $\varphi$ \dfn{real} if $\varphi(H)^*=\varphi(H^*).$
We say that $\varphi$ is \dfn{completely positive} if each $\varphi^{(n)}$ takes positive semidefinite elements to positive semidefinite elements.

Let $X \in M_n(\mathbb{C}).$
Define the Lyapunov map $L_X(H) = H- X^*HX.$

\begin{theorem}[Nevanlinna-Pick interpolation theorem: Lyapunov formulation]\label{lyaclassic}
	Let $Z \in M_n(\mathbb{C})$ be a strict contraction and $\Lambda \in M_n(\mathbb{C}).$
	There is an analytic function $f: \mathbb{D} \rightarrow \cc{\mathbb{D}}$
	such that $f(Z)=\Lambda$ if and only if $L_\Lambda \circ L_Z^{-1}$ is a completely positive map.
\end{theorem}
We give a proof of the above theorem which demonstrates our technique in Subsection \ref{techdemo}.

Some work shows that if $Z$ is a matrix with $z_i$ on the diagonal and $\Lambda$ with $\la_i$ on the diagonal,
the corresponding Choi matrix is exactly the matrix arising in the classical Nevanlinna-Pick interpolation theorem.
The formulation of the problem in terms of complete positivity of some induced map is a powerful idea which has led to broad generalizations
in noncommuting variables, especially in terms of the work of Ball-Groenwald-Malakorn \cite{bgm1,bgm2} and subsequent developments \cite{bmv1, bmv2, bmv3, pascoeinv}.

We seek to understand the lurking isometry argument and consequent transfer function realization type objects as fundamental rather than coincidental. That is, we analyze the relationship between two real completely bounded maps $\varphi$ and $\psi$ such that
$\psi \circ \varphi^{-1}$ is completely positive. This framework captures much of operator theoretic interpolation theory arising from generalizations of Nevanlinna-Pick interpolation. The Stinesping theorem  on factorization of completely positive maps envelops some of the core ideas of operator theory as special cases, including the GNS construction, Choi's theorem, and the Sz.-Nagy dilation theorem \cite{paulsen, pisier}. The induced Stinespring type theorem likewise envelops and extends Nevanlinna-Pick type interpolation theorems and related realization machinery.

\section{Stinespring factorization, Wittstock decompostion, and support}

The \dfn{Stinespring factorization theorem} \cite{stinespring} states that any completely positive map $\varphi: \mathcal{M}\rightarrow \BH$
is of the form 
	$$\varphi(H)=V^*\pi(H)V$$
where $\pi:\mathcal{M}\rightarrow \mathcal{B}(\mathcal{L})$ is a representation of $\mathcal{M}$ on a Hilbert space $\mathcal{L}$
and $V:\mathcal{H} \rightarrow \mathcal{L}$ is a bounded linear operator. We say two representations are \dfn{equivalent} if they are unitarialy similar. We say $\pi_1$ is a \dfn{subrepresentation} of $\pi_2$ if
$\pi_2$ is equivalent to representation of the form $\pi_1 \oplus \pi_3.$ 

Given $\pi_1$ and $\pi_2$ representations, we say $\pi_1$ and $\pi_2$ are \dfn{totally orthogonal}, denoted $\pi_1 \perp \pi_2,$ if there does not exist
a subrepresentation of $\pi_1$ which is equivalent to a subrepresentation of $\pi_2.$
The notion of total orthogonality gives a generalization of Schur's lemma which holds even in the absence of a notion of irreducibility.
\begin{observation}
Let $\pi_1$ and $\pi_2$ be representations of some $C^*$-algebra $\mathcal{M}.$
If $\pi_1 \perp \pi_2$ and $A$ is an operator such that $\pi_1(m)A = A\pi_2(m)$ for all $m \in \mathcal{M}$
then $A=0.$
\end{observation}
We say $\supp \pi_1 \subseteq  \supp \pi_2$ 
if there does not exist a subrepresentation of $\pi_1$ which is totally orthogonal to $\pi_2.$ We say $\supp \pi_1 = \supp \pi_2$
if $\supp \pi_1 \subseteq \supp \pi_2$ and $\supp \pi_2 \subseteq \supp \pi_1.$ We call the symbol $\supp \pi$ the \dfn{support} of $\pi.$
%Note that $\supp \pi_1 = \supp \pi_2$ if and only if there exists some Hilbert space $\mathcal{L}$ such that $I_{\mathcal{L}} \otimes \pi_1$ is equivalent $I_{\mathcal{L}} \otimes \pi_2.$
If $\varphi$ is a completely positive map, we define the $\supp \varphi$ to be the support of the corresponding representation 
in its minimal Stinespring factorization.

The \dfn{Wittstock decomposition theorem} \cite{wittstock, paulsen82} states that any completely bounded map is in the span of the completely positive maps.

Thus, any real completely bounded map $\varphi$ can be decomposed as
	$$\varphi = \varphi^+ - \varphi^-.$$
We call such a Wittstock decomposition \dfn{extremal} if there does not exist a completely positive $\delta$
such that $\varphi^+-\delta$ and $\varphi^--\delta$ are completely positive.
An extremal Wittstock decomposition is similar to the Hahn-Jordan decomposition of measures, though an extremal Wittstock decomposition is not always unique.

However, we prove that the support of a Wittstock decomposition is unique in the following sense.
\begin{theorem}[Wittstock support theorem]\label{Wittstocksupport}
	Let $\mathcal{M}$ be a $C^*$-algebra.
	Suppose $\varphi:\mathcal{M}\rightarrow \BH$ be a real completely bounded map.
	Given two extremal Wittstock decompositions $$\varphi = \varphi^+-\varphi^- = \phi^+-\phi^-,$$
	we have that
		$$\supp \varphi^+ = \supp \phi^+, \supp \varphi^- = \supp \phi^-.$$
\end{theorem}
We prove Theorem \ref{Wittstocksupport} at the end of Section \ref{agord}. Related uniqueness statements about generalized Stinespring representations have been recently obtained by Christensen in \cite[Theorem 3.1]{christensen}, who shows that maps of the form $\varphi(H)=W^*\pi(H)V$ have support uniqueness with respect to $\pi$ under controllability-observability minimality type assumptions. Our results show that in an extremal Wittstock decomposition, one cannot introduce extraneous representations.

\section{The Agler order and colligations} \label{agord}
Let $\mathcal{M}, \mathcal{N}, \tilde{\mathcal{N}}$ be $C^*$-algebras.
Let $\varphi: \mathcal{M} \rightarrow \mathcal{N}$ and 
	$\psi: \mathcal{M} \rightarrow \tilde{\mathcal{N}}$ be real completely bounded maps.
We say that $\varphi \leq \psi$ in the \dfn{Agler order} if:
\begin{enumerate}
	\item $\ker \varphi \subseteq \ker \psi,$
	\item the induced map $\gamma = \psi \circ \varphi^{-1}$ is completely positive.
\end{enumerate}
We say $\varphi$ is \dfn{Archimedian} if its range contains a strictly positive element.

The Agler order captures the various Agler models used for Nevanlinna-Pick interpolation in the Schur-Agler, Herglotz-Agler, Pick-Agler classes and so on, codifying the Lyapunov formulation taken in \cite{bgm1,bgm2, bmv1, bmv2, bmv3, pascoeinv}. 
For example, taking $\varphi(H)=H- X^*HX, \psi(H) = Y^*H+HY,$ we have that $\varphi \leq \psi$ corresponds to there being an analytic function from the disk to the right half plane (a Herglotz function) taking $X$ to $Y.$ (Similarly for noncommutative and commutative multivariable analogues, cf. \cite{ag90, pptdHerglotz}.) The case $\varphi(H)=(X^*H- HX)/2i, \psi(H) = Y^*H-HY/2i$
similarly corresponds to the existence of a Nevanlinna model for a Pick function as in \cite{ptdpick}.
The Agler order abstracts away the domain and range conditions, and interpolation interpretation for the more basal underlying condition of induced complete positivity. 
%\red might say a tiny bit more here if it doesn't muddy the waters \black

Let $\mathcal{M}, \mathcal{N}, \tilde{\mathcal{N}}$ be $C^*$-algebras where $\mathcal{N}, \tilde{\mathcal{N}}$ are concrete.
Let $\varphi: \mathcal{M} \rightarrow \mathcal{N}$ and 
	$\psi: \mathcal{M} \rightarrow \tilde{\mathcal{N}}$ be real completely bounded maps.
We say that $\varphi \preceq \psi$ in the \dfn{concrete Agler order} if there exists an operator $\Gamma$
such that $\Gamma^* \varphi(H)\Gamma = \psi(H).$

We note that if $\varphi$ is Archimedian that $\varphi \leq \psi$ (in the Agler order) if and only if
for every representation $\pi$ of $\tilde{\mathcal{N}}$ there exists a representation $\hat{\pi}$ of
$\mathcal{N}$ such that $\hat{\pi}\circ \varphi \preceq \pi \circ \psi.$

	Suppose $\varphi$ and $\psi$ are real completely bounded maps into concrete $C^*$-algebras.
	We say $\psi$ is \dfn{$\varphi$-colligatory} if for all Wittstock decompostions
		$$\varphi = \varphi^+ - \varphi^-,$$ $$\psi = \psi^+ - \psi^-,$$
	%for each representation $\pi$ of $\mathcal{N}'$
	%there exists a representation $\hat{\pi}$ of $\mathcal{N}$
	given Stinespring factorizations
		$$\psi^+ = \Psi_+^* \pi_1 \Psi_+,$$ $$\psi^- = \Psi_-^* \pi_2 \Psi_-,$$
		$$\varphi^+ = \Phi_+^* \pi_3 \Phi_+,$$ $$\varphi^- = \Phi_-^* \pi_4 \Phi_-,$$
	there is a partial isometry
		$$U= \bbm A & B \\ C & D \ebm$$
	and operator $\Gamma$ such that
		$$ \bbm \Psi_- \\ \Phi_+ \Gamma \ebm
		=\bbm A & B \\ C & D \ebm \bbm \Psi_+ \\ \Phi_- \Gamma \ebm,$$
	where $$\ran U = \vspan \bigcup_{H\in \mathcal{M}} \ran \bbm \pi_2(H) \Psi_- \\ \pi_3(H) \Phi_+ \Gamma \ebm$$
	and $$\ran U^* = \vspan \bigcup_{H\in \mathcal{M}} \ran \bbm \pi_1(H)\Psi_+ \\ \pi_4(H) \Phi_- \Gamma \ebm,$$
	and 
		$$\bbm A & B \\ C & D \ebm \bbm \pi_1 &  \\ & \pi_4 \ebm =  \bbm \pi_2 &  \\ & \pi_3 \ebm \bbm A & B \\ C & D \ebm.$$
	Note that, if $\Phi_+ - D\Phi_-$ is invertible, then
		$$\Psi_- = [A+B\Phi_-(\Phi_+ - D\Phi_-)^{-1}C]\Psi^{+}.$$
	We call such an expression a \dfn{$\varphi$-transfer function realization}.
	We call $U$ the \dfn{colligation operator}.

	We prove the following concrete result.
	\begin{theorem}\label{concretecoll}
		Let $\varphi$ and $\psi$ be real completely bounded maps on some $C^*$-algebra $\mathcal{M}$
		mapping into concrete $C^*$-algebras.
		%Suppose $\varphi$ is Archimedian.
		The following are equivalent:
		\begin{enumerate}
			\item $\varphi \preceq \psi$ in the concrete Agler order,
			\item $\psi$ is $\varphi$-colligatory.
		\end{enumerate}
	\end{theorem}
	\begin{proof}
	Take Wittstock decompostions
		$$\varphi = \varphi^+ - \varphi^-,$$ $$\psi = \psi^+ - \psi^-,$$
	and Stinespring factorizations
		$$\psi^+ = \Psi_+^* \pi_1 \Psi_+,$$ $$\psi^- = \Psi_-^* \pi_2 \Psi_-,$$
		$$\varphi^+ = \Phi_+^* \pi_3 \Phi_+,$$ $$\varphi^- = \Phi_-^* \pi_4 \Phi_-.$$
	Since $\psi = \Gamma^* \varphi\Gamma,$
	we see that $$\Psi_+^* \pi_1 \Psi_+ - \Psi_-^* \pi_2 \Psi_-
	= \Gamma^*\left[\Phi_+^* \pi_3 \Phi_+ -\Phi_-^* \pi_4 \Phi_- \right] \Gamma$$
	Rearranging, we get 
	$$\Psi_+^* \pi_1 \Psi_+ -\Gamma^*\Phi_-^* \pi_4 \Phi_-\Gamma
	= \Gamma^*\Phi_+^* \pi_3 \Phi_+ \Gamma + \Psi_-^* \pi_2 \Psi_-$$
	Evaluating at $W^*V$ gives
		$$\Psi_+^* \pi_1(W^*V) \Psi_+ + \Gamma^*\Phi_-^* \pi_4(W^*V) \Phi_-\Gamma
	=  \Psi_-^* \pi_2(W^*V) \Psi_- + \Gamma^*\Phi_+^* \pi_3(W^*V) \Phi_+ \Gamma.$$
	So,
		$$\Psi_+^* \pi_1(W)^*\pi_1(V) \Psi_+ + \Gamma^*\Phi_-^* \pi_4(W)^*\pi_4(V) \Phi_-\Gamma$$
	is equal to
	$$\Psi_-^* \pi_2(W)^*\pi_2(V) \Psi_- +  \Gamma^*\Phi_+^* \pi_3(W)^*\pi_3(V) \Phi_+ \Gamma.$$
	Factoring, we get that for vectors $v, w,$
	$$\left\langle \bbm \pi_1(V) \Psi_+ \\ \pi_4(V) \Phi_-\Gamma \ebm v, \bbm \pi_1(W) \Psi_+ \\ \pi_4(W) \Phi_-\Gamma \ebm w \right\rangle
	=\left\langle \bbm \pi_2(V) \Psi_- \\ \pi_3(V) \Phi_+\Gamma \ebm v, \bbm \pi_2(W) \Psi_- \\ \pi_3(W) \Phi_+\Gamma \ebm w \right\rangle.$$
	So, there is a partial isometry
		$$U= \bbm A & B \\ C & D \ebm$$
	and $\Gamma$ such that
		$$ \bbm \Psi_- \\ \Phi_+ \Gamma \ebm
		=\bbm A & B \\ C & D \ebm \bbm \Psi_+ \\ \Phi_- \Gamma \ebm,$$
	where $$\ran U = \vspan \bigcup_{H\in \mathcal{M}} \ran \bbm \pi_2(H) \Psi_- \\ \pi_3(H) \Phi_+ \Gamma \ebm$$
	and $$\ran U^* = \vspan \bigcup_{H\in \mathcal{M}} \ran \bbm \pi_1(H)\Psi_+ \\ \pi_4(H) \Phi_- \Gamma \ebm,$$
	and 
		$$\bbm A & B \\ C & D \ebm \bbm \pi_1 &  \\ & \pi_4 \ebm =  \bbm \pi_2 &  \\ & \pi_3 \ebm \bbm A & B \\ C & D \ebm.$$
	%{\red last part needs some proof}
	\end{proof}

	We see the immediate corollary for the Agler order.
	\begin{corollary}\label{inconcrete}
		Let $\varphi$ and $\psi$ be real completely bounded maps on some $C^*$-algebra $\mathcal{M}$
		mapping into $C^*$-algebras $\mathcal{N}, \tilde{\mathcal{N}}.$
		Suppose $\varphi$ is Archimedian.
		The following are equivalent:
		\begin{enumerate}
			\item $\varphi \leq \psi$ in the Agler order,
			\item For every representation $\pi$ of $\tilde{\mathcal{N}}$
			there exists a representation $\hat{\pi}$ of $\mathcal{N}$ such that $\pi\circ\psi$ is $\hat{\pi}\circ\varphi$-colligatory.
		\end{enumerate}
	\end{corollary}
	
	Say a real completely bounded map $\varphi$ is of \dfn{Lyapunov type}
	if it admits a Wittstock decomposition $\varphi = \pi - \varphi^-$
	where $\pi$ is a representation and $\varphi^-$ is strictly completely contractive.

	Note that any $\varphi$ of Lyapunov type is \emph{a fortiori} Archimedian.
	\begin{corollary}\label{lyapunovtype}
		Let $\varphi$ and $\psi$ be real completely bounded maps on some $C^*$-algebra $\mathcal{M}$
		mapping into concrete $C^*$-algebras.
		Suppose $\varphi$ is of Lyapunov type.
		The following are equivalent:
		\begin{enumerate}
			\item $\varphi \preceq \psi$ in the concrete Agler order,
			\item $\psi$ is $\varphi$-colligatory,
			\item $\psi$ has a $\varphi$-transfer function realization.
		\end{enumerate}
	\end{corollary}

	\begin{observation}\label{factoralong}
	We also note that if $\pi_i = \bbm \hat{\pi_i} & \\ & \tilde{\pi}_i\ebm$
	such that $\supp \hat{\pi}_i \perp \supp \tilde{\pi}_j,$ then the colligation operator factors 
	as
		$$\bbm \hat{A} & & \hat{B} & \\ & \tilde{A} & & \tilde{B} \\ \hat{C} & & \hat{D} & \\ & \tilde{C} & & \tilde{D}\ebm.$$
	Thus, $$\pi\circ{\psi}^\pm = \hat{\psi}^\pm + \tilde{\psi}^\pm$$ where $\supp \hat{\psi}^{\pm} \perp \supp \tilde{\psi}^{\pm},$
	and $$\hat{\pi}\circ{\varphi^{\pm}} = \hat{\varphi}^{\pm} + \tilde{\varphi}^{\pm}$$
	where $\supp \hat{\varphi}^{\pm} \perp \supp \tilde{\varphi}^{\pm}$ and, letting $\hat{\varphi} = \hat{\varphi}^+ -\hat{\varphi}^-$
	and $\hat{\psi} = \hat{\psi}^+ -\hat{\psi}^-,$
	$$\hat{\varphi}\preceq \hat{\psi}, \tilde{\varphi}\preceq \tilde{\psi}.$$
	\end{observation}
	
	\begin{proof}[Proof of Theorem \ref{Wittstocksupport}]
		We consider the case of the positive supports. The negative case is similar.

		Let $\varphi$ have two Wittstock decompositions
		$$\varphi = \varphi^+ - \varphi^- = \Psi_+ - \Psi_-,$$
			and Stinespring factorizations
		$$\psi^+ = \Psi_+^* \pi_1 \Psi_+,$$ $$\psi^- = \Psi_-^* \pi_2 \Psi_-,$$
		$$\varphi^+ = \Phi_+^* \pi_3 \Phi_+,$$ $$\varphi^- = \Phi_-^* \pi_4 \Phi_-.$$
		By Theorem \ref{concretecoll}, there is a colligation operator such that
			$$ \bbm \pi_2 \Psi_- \\ \pi_3 \Phi_+  \ebm
		=\bbm A & B \\ C & D \ebm \bbm \pi_1 \Psi_+ \\ \pi_4 \Phi_- \ebm$$
		Write $$\pi_i = \bbm\hat{\pi_i}&\\&\tilde{\pi_i}\ebm$$ 
		where $$\supp \pi_1 \perp \supp \tilde{\pi_i}.$$
		By Observation \ref{factoralong}, factor the colligation operator as, noting that there is no $\tilde{\pi_1},$
		$$\bbm \hat{\pi_2} \hat{\Psi_-} \\ \tilde{\pi_2} \tilde{\Psi_-}\\ \hat{\pi_3} \hat{\Phi_+} \\ \tilde{\pi_3} \tilde{\Phi_+} \ebm
		=  \bbm \hat{A} & \hat{B} & \\ & & \tilde{B} \\ \hat{C} & \hat{D} & \\ & & \tilde{D}\ebm
		\bbm \pi_1 \Psi_+ \\ \hat{\pi_4}  \hat{\Phi_-} \\ \tilde{\pi_4} \tilde{\Phi_-} \ebm.$$
		
		So, we have that
			$$\bbm \tilde{\pi_2} \tilde{\Psi_-}\\ \tilde{\pi_3} \tilde{\Phi_+} \ebm
		=  \bbm  0& \tilde{B}  \\ 0 & \tilde{D}\ebm
		\bbm 0 \\ \tilde{\pi_4} \tilde{\Phi_-} \ebm.$$
		Hence, by Theorem \ref{concretecoll},
		we see that 
		$$- \tilde{\Psi_-}^*\tilde{\pi_2}\tilde{\Psi_-} = \tilde{\Phi_+}^* \tilde{\pi_3} \tilde{\Phi_+} -
		\tilde{\Phi_-}^* \tilde{\pi_4} \tilde{\Phi_-}.$$
		If $\tilde{\Phi_+}^* \tilde{\pi_3} \tilde{\Phi_+} \neq 0$,
		taking $\varphi^+ - \tilde{\Phi_+}^* \tilde{\pi_3} \tilde{\Phi_+}$
		and $\varphi^- - \tilde{\Phi_+}^* \tilde{\pi_3} \tilde{\Phi_+}$
		witnesses the nonextremality of the Wittstock decomposition.
	\end{proof}

\section{Truncating irrelevant representations and the commutant coefficient theorem}
	The following proposition shows that one can choose a natural tensored representation in the Stinespring representation of a (homomorphic) noncommutative conditional expectation. Given $\mathcal{N} \subseteq \BH$ we use $\mathcal{N}'$ to denote the \dfn{commutant} of $\mathcal{N},$ the set of elements in $\BH$ which commute with every element of $\mathcal{N}.$
	\begin{proposition}
	Let $\mathcal{M}\subseteq \mathcal{B}(\mathcal{L})$ be a $C^*$-algebra.
	Let $\mathcal{N}$ be a sub-$C^*$-algebra unitally included in $\mathcal{M}$
	such that $\mathcal{M}$ is generated by $\mathcal{N}$ and $\mathcal{N}'.$ 
	Let $\pi: \mathcal{N}\rightarrow \BH.$
	Let $E: \mathcal{H} \rightarrow \mathcal{L}$ such that $E^*nE=\pi(n).$
	There is a representation
	$\hat{\pi}:\mathcal{M}\rightarrow \mathcal{B}(\mathcal{H}\otimes \mathcal{K}),$ a unit vector $e_0 \in \mathcal{K}$ and a partial isometry $P: \mathcal{L}\rightarrow \mathcal{H}\otimes \mathcal{K}$ with range containing all vectors of the form $v\otimes e_0$ such that 
		$$E^*mE = (I\otimes e_0^*)\hat{\pi}(m)(I\otimes e_0),$$
		$$P m = \hat{\pi}(m)P,$$
	and $$\hat{\pi}|_\mathcal{N} =  \pi\otimes I.$$
\end{proposition}
\begin{proof}
	 We will show that there is a reducing subspace $\mathcal{L}'$ for $\mathcal{M}$
	 elements of $\mathcal{N}$ 
	are unitarialy equivalent to $\bigoplus P^*\pi(n)P$ on $\mathcal{L}'$
	for some isometries $P$ which reduce $\mathcal{N},$ and at least one $P$ is the identity.
	Taking a sufficiently big direct sum of $\mathcal{M}$ reduced to $\mathcal{L}'$
	by some kind of Hilbert hotel argument. (That is, pairing up incomplete representations
	$P^*\pi(n)P$ with their orthogonal complement.) 
	
	We will show that given $m \in \mathcal{N}'$
	and $\mathcal{L}_0$ containing $\mathcal{H}$ such that $\mathcal{N}$ has the desired form
	and $\mathcal{L}_0 \neq m\mathcal{H} + \mathcal{L}_0$
	then there is a larger $\mathcal{L}_1$ such that $\mathcal{N}$ has the desired form. Let $\hat{m}=P_{\mathcal{L}_0^\perp}mP_{\mathcal{H}}$
	Note $nm = mn.$ So, since $\mathcal{L}_0$ reduces $n,$
	$$n\hat{m} = \hat{m}n.$$
	Thus, $\hat{m}^*n\hat{m} = \hat{m}^*\hat{m}n=n\hat{m}^*\hat{m}.$
	Letting $P=\hat{m}(\hat{m}^*\hat{m})^{\dagger 1/2}.$
	So $P^*nP=P_{\ran \hat{m}^*}^*nP_{\ran \hat{m}^*}.$ Note
	$P_{\ran \hat{m}^*}^*nP_{\ran \hat{m}^*}$ is a subrepresentation of
	$\pi.$
	
	So there is a minimal $\mathcal{L}$ such that 
	$\mathcal{L} = m\mathcal{H} + \mathcal{L}$
	for all $m$ and $\mathcal{L}$ reduces $\mathcal{N}.$
	If fact, $\mathcal{L}$ is exactly the minimal reducing subspace for $\mathcal{M}$ containing $\mathcal{H}.$
\end{proof}

	Let $\pi$ be a representation of a $C^*$-algebra $\mathcal{M}$. 
	Call a complete bounded map $\psi$ \dfn{$\pi$-pure} if %\red ask about this\black
	$$\psi =  {\Psi_+}^*(I \otimes \pi)\Psi_+  - {\Psi_-}^* (I \otimes \pi) \Psi_-.$$
	We see the following immediate corollary of the above representation theorem.
	\begin{corollary}\label{commrep}
		Let $\varphi$ and $\psi$ be real completely bounded on some $C^*$-algebra $\mathcal{M}.$
		Assume $\psi$ is $\pi$-pure.
		The following are equivalent:
		\begin{enumerate}
			\item $\varphi \leq \psi$ in the Agler order,
			\item There exists a representation $\hat{\pi}$ such that
			$\hat{\pi}\circ \varphi$ is $\pi$-pure
			and $\psi$ is $\hat{\pi}\circ \varphi$-colligatory.
		\end{enumerate}
	\end{corollary}
	The compatibility of representations is important in infinite dimensional noncommutative interpolation problems, where there 
	is some work to show the abstract technique here gives a {\it bona fide} solution (see, for example, the last section of \cite{pascoeinv}).

\section{Examples from interpolation theory}

\subsection{Nevalinna-Pick interpolation in a Lyapunov formulation} \label{techdemo}

%Now, let $X \in \Mn$ with $\norm{X} < 1$ and let $Y \in \Mn$ so that the positivity condition
%\beq\label{eq-lyupos}
%L_X(A) \geq 0 \Rightarrow L_Y(A) \geq 0
%\eeq
%holds. 
The following example demonstrates how our method works in the classical case.

Given $\norm{X} \leq 1$, the Lyapunov map $L_X$ is invertible, and
\[
L_X\inv(H) = \sum_{n=0}^\infty {X\ad}^n H X^n.
\]
Define the operator 
\[
\Lambda_{XY} = L_Y \circ L_X\inv.
\]
Theorem \ref{lyaclassic} says that $\Lambda_{XY}$ must be completely positive for the corresponding interpolation problem to be solvable.
%With this definition, the positivity condition is equivalent to $\Lambda_{XY}$ being a completely positive map in the sense that if $A \in \Mnm$ and $A \geq 0$ then $\Lambda_{XY}(A) \geq 0$.

As an example, consider the choice of matrices
\[
X = \bbm z_1 & & \\ & \ddots & \\ & & z_n\ebm, \hspace{1cm} Y = \bbm \la_1 & & \\ & \ddots & \\ & & \la_n\ebm
\]
Computing the explicit form of $L_X\inv$, we get
\begin{align*}
L_X\inv(H) &= \sum_{n=0}^\infty {\bbm z_1 & & \\ & \ddots & \\ & & z_n\ebm\ad}^n (h_{ij})_{i,j} \bbm z_1 & & \\ & \ddots & \\ & & z_n\ebm^n\\
&= \left(\frac{h_{ij}}{1 - \cc{z}_i z_j}\right)_{i,j}
\end{align*}
Plugging in $Y$ to $L_Y$ gives
\[
L_Y(H) = H - \bbm \la_1 & & \\ & \ddots & \\ & & \la_n\ebm\ad H \bbm \la_1 & & \\ & \ddots & \\ & & \la_n\ebm,
\]
and so 
\beq\label{eq-hpick}
\Lambda_{XY}(H) = L_Y\circ L_X\inv(H) = \left( h_{ij} \frac{1 - \cc\la_i \la_j}{1 - \cc z_i z_j} \right)_{i,j}.
\eeq
In this case, the positivity condition becomes 
\[
H = (h_{ij}) \geq 0 \Rightarrow \left( h_{ij} \frac{1 - \cc\la_i \la_j}{1 - \cc z_i z_j} \right)_{i,j} \geq 0,
\]
which recovers the Pick condition
\[
\left(\frac{1 - \cc\la_i \la_j}{1 - \cc z_i z_j} \right)_{i,j} \geq 0.
\]
That is, we have recast classical Nevanlinna-Pick interpolation as a question about completely positive maps.

\subsubsection{Lurking isometries}

When $\Lambda_{XY}$ is a completely positive map, the Stinespring theorem allows us to write $\Lambda_{XY}$ by
\[
\Lambda_{XY}(H) = \Gamma\ad \pi(H) \Gamma
\]
where $\pi:\Mn \to B(\hilbert)$ is a representation.
\begin{note}
In this special case where $H \in \Mn$, we know the homomorphisms. There are no closed ideals. All representations of $\Mn$ are the same. So write
\[
\pi(H) = I\otimes H
\]
%(though we'll follow the usual Helton convention on tensor products).
\end{note}

Since
\[
L_Y \circ L_X\inv(H) = \Lambda_{XY}(H),
\]
we can calculate
\begin{align*}
L_Y(H) &= \Lambda_{XY} \circ L_X (H) \\
H - Y\ad H Y &= \Gamma\ad \pi(H - X\ad H X) \Gamma \\
H + \Gamma\ad \pi(H) \Gamma &= Y\ad H Y + \Gamma\ad \pi(X\ad H X) \Gamma.
\end{align*}

After setting $H = W\ad V$ (and conjugation by $\alpha, \beta$), we get
\[
W\ad V + \Gamma\ad \pi(W\ad) \pi(V) \Gamma = Y\ad W\ad V Y + \Gamma\ad \pi(X)\ad \pi(W\ad)\pi(V) \pi(X) \Gamma.
\]

This is the setup for the so-called lurking isometry argument. That is, there exists $U = \bbm A & B \\ C& D\ebm$ so that 
\beq\label{lurk}
\bbm A & B \\ C & D \ebm \bbm V \\ \pi(V) \Gamma \ebm = \bbm VY \\ \pi(V)\pi(X) \Gamma \ebm,
\eeq
where $U$ is a partial isometry (that is, $(U\ad U)^2 = U\ad U $).

Furthermore,
\[
\bbm A & B \\ C & D \ebm \bbm V & \\ & \pi(V) \ebm = \bbm V & \\ & \pi(V) \ebm \bbm A & B \\ C & D \ebm.
\]
Note that A, B, C, D factor as $A = \hat{A}\otimes I,$ $B = \hat{B}\otimes I,$ and so on.

Now set $V = I$. Then \eqref{lurk} becomes
\[
 \bbm Y \\ \Gamma \ebm = \bbm A & B \\ C & D \ebm \bbm I  \\  \pi(X) \Gamma \ebm 
\]
leading to the equations
\begin{align*}
Y &= A + B \pi(X) \Gamma \\
\Gamma &= C + D \pi(X) \Gamma.
\end{align*}

Eliminating $\Gamma$ gives the (Schur-Agler) transfer function realization
\beq\label{eq-tfr}
Y = A + B \pi(X) (1 - D \pi(X))\inv C,
\eeq
which points to the existence of an interpolating  function in terms of complete positivity. (c.f. \cite{bv2005, bv2007, pascoeinv})

\subsection{Two variable commutative Nevanlinna-Pick interpolation}
Let $X_1, X_2$ be commuting matrices. Define
\[
\ph = L_X(H) = \bbm H & \\ & H \ebm - \bbm X_1\ad H X_1 & \\ & X_2\ad H X_2 \ebm
\]
and
\[
\psi = L_Y(H) = H - Y\ad H Y.
\]
This setup gives the 2-variable commutative Nevanlinna-Pick interpolation theorem.
Similarly, as in \cite{pascoeinv}, our method works on more general semi-algebraic sets in the general noncommutative case.

%\subsection{Nevanlinna-Pick interpolation on a polynomially convex set $B_\delta$}
%Let
%\[
%\ph = H - \delta(X)\ad H \delta(X)
%\]
%where $\delta$ is the typical sort of noncommutative function used to define ball-like objects \red figure out the shortest way to write this... nc polynomial or whatever \black. Let 
%\[
%\psi = L_Y = H - Y\ad H Y.
%\]
%This is interpolation on $B_\delta$ \red which needs a def when this gets fleshed out \black.

\subsection{Partial Nevanlinna-Pick interpolation}
Let
\[
\ph = H - X\ad H X
\]
and
\[
\psi = w\ad H w - v\ad H v.
\]

Notice that $\psi$ is scalar-valued, and so positivity of $\psi \circ \ph\inv$ implies complete positivity in this case. In the notation of this section, we have $\Psi_+ = w$ and $\Psi_- = v$, as well as $\Phi_+ = I$ and $\Phi_- = X$. By Theorem \ref{lyapunovtype} we get the transfer function formulation 
\begin{align*}
v &= [A + BX(I - DX)\inv C]w \\
&= f(X) w
\end{align*}
using the usual Schur-Agler representation $f(X) = A + BX(I - DX)\inv C$.

If we make the definitions
\[
X = \bbm z_1 & & \\ & \ddots & \\ & & z_n \ebm, w = \bbm 1 \\ \vdots \\ 1\ebm, v = \bbm \la_1 \\ \vdots \\ \la_n\ebm,
\]
the problem becomes to look at the existence of a function $f$ so that $f(z_i) = \la_i$.
The technique here generalizes to other settings of solving $f(X)v=w,$ including noncommutative Nevanlinna-Pick interpolation.

\subsection{Commutant coefficient interpolation}
Let $\mathcal{M}$ be a $C^*$-algebra.
Consider $L_X(H) = H - X^*HX$ as a map from $\mathcal{M}$ to itself for some $X \in \mathcal{M}.$
We see by Corollary \ref{commrep} that if $L_X \leq L_Y$ and $X$ is strictly contractive,
then
	$$Y = A+B(I \otimes X)(1-DX)^{-1}C=A + \sum BD^nC X^{n+1}.$$

\bibliography{references}
\bibliographystyle{plain}

%\printindex

\end{document}